\documentclass[12pt]{amsart}
\usepackage{mathrsfs}
\usepackage{geometry}
\usepackage[all]{xy}

\usepackage{amsmath, amsthm}
\usepackage{amssymb}
\usepackage{amscd}
\usepackage{latexsym}
\usepackage{epsfig}
\usepackage{graphics}
\usepackage{amsfonts}
\usepackage{psfrag}
\usepackage{graphicx}
\usepackage{subfigure}
\usepackage{fancyhdr}

\def\N {{\mathbb{N}}}

\def\R {{\mathbb{R}}}

\def\p {{\bf{p}}}

\def\ra{{\rightarrow}}
\def\ba{{\backslash}}

\newcommand{\ds}{\displaystyle}

\newtheorem{theorem}{Theorem}[section]
\newtheorem{lemma}[theorem]{Lemma}
\newtheorem{proposition}[theorem]{Proposition}
\newtheorem{corollary}[theorem]{Corollary}

\theoremstyle{definition}
\newtheorem{definition}[theorem]{Definition}
\newtheorem{example}[theorem]{Example}
\newtheorem{question}[theorem]{Question}
\newtheorem{remark}[theorem]{Remark}
\newtheorem{remark on notation}[theorem]{Remark on Notation}
\newtheorem{notation}[theorem]{Notation}

\numberwithin{equation}{section}

\begin{document}
\title[Low Bound For the Rank of Rigidity Matrix of 4-valent graphs]{lower bound for the rank of rigidity matrix of 4-valent graphs under various connectivity assumptions}

\author{Shisen Luo}

\address{Department of Mathematics, Cornell University,
Ithaca, NY 14853-4201, USA}

\email{{\tt ssluo@math.cornell.edu}}


\begin{abstract}
In this paper we study the rank of planar rigidity matrix of 4-valent graphs, both in case of generic realizations and configurations in general position, under various connectivity assumptions on the graphs. For each case considered, we prove a lower bound and provide an example which shows the order of the bound we proved is sharp. This work is closed related to work in \cite{Luo:Rigidity} and answers some questions raised there.
\end{abstract}

\date{\today}
\maketitle \tableofcontents

%%%%%%%%%%%%%%%%%%%%%%%%%%%%%%%%%%%%%%%%%%%%%%%%%%%
\section{\bf Introduction}\label{sec:introduction}
Let $G=(V, E)$ be a connected 4-valent graph and $\p: V\ra \R^2$ a planar realization(the words realization and configuration will be used interchangeably).  The graph $G$ is always assumed to be finite and simple. The rigidity matrix, which we denote by $R(\p)$, is a matrix of size $|E|\times 2|V|$. The rank of $R(\p)$, which we denote by $r(G(\p))$, can be taken as the definition of the rank of the {\it infinitesimal rigidity matroid} of the framework $G(\p)$.  When $\p$ is a {\it generic realization}, one can show $r(G(\p))$ is independent of $\p$ as long as it is generic, this can be taken as the definition of the rank of the {\it generic rigidity matroid} of $G$ and we will write $r(G)$ for it. For more details on the definitions and terminologies in rigidity theory, we refer the readers to \cite{GSS:Combinatorial Rigidity}.

As we pointed out in \cite{Luo:Rigidity}, the question of finding the lower bound of the rigidity matrix is closed related to a question in symplectic geometry. We will not repeat it here but refer the readers to \cite{Luo:Rigidity} and references therein.

The main results in the paper are summarized as the following four theorems. 

\begin{theorem}\label{thm:generic}
Assume $G=(V,E)$ is a connected 4-valent graph, then \[r(G)\geq \dfrac{8}{5}|V|-1.\]
\end{theorem}

\begin{theorem}\label{thm:general}
Assume $G=(V,E)$ is a connected 4-valent graph and $\p: V\ra \R^2$ is a configuration in general position, i.e., the image under $\p$ of any three points in $V$ do not lie on the same line, then $$r(G(\p))\geq \dfrac{8}{5}|V|-1.$$
\end{theorem}

\begin{theorem}\label{thm:connected generic}
Assume $G=(V,E)$ is a connected 4-valent graph which is also $4$-edge-connected (see Definition~\ref{def:edge-connected}), then
\[r(G)\geq \dfrac{7|V|-7}{4}.\]
\end{theorem}

\begin{theorem}\label{thm:connected general}
Assume $G=(V,E)$ is a connected 4-valent graph which is also $4$-edge-connected and $\p: V\ra \R^2$ is a configuration in general position, then
\[r(G(\p))\geq \dfrac{5|V|-4}{3}.\] 
\end{theorem}

\begin{definition}\label{def:edge-connected}
A graph $G=(V, E)$ is called $k$-edge-connected for some $k\in \N$ if $G$ remains connected upon removing any $k-1$ edges. In particular, a connected graph is $1$-edge-connected.
\end{definition}

Theorem~\ref{thm:generic} was Theorem 1 in \cite{Luo:Rigidity}, here we will give an alternative proof of it by first proving Theorem~\ref{thm:connected general} and then Theorem~\ref{thm:general}. Theorem~\ref{thm:general}, Theorem~\ref{thm:connected general} and Theorem~\ref{thm:connected generic} provide answers to Question 1.9 and Question 1.10 in \cite{Luo:Rigidity} in case of 4-valent graphs. It should also be pointed out that Theorem~\ref{thm:connected general} is essentially proved in Section 5 of \cite{Luo:Betti}, it is the $``d=1"$ version of Theorem 5.1(or Proposition 5.7) there. However, \cite{Luo:Betti} was written in a language (mathematically) different from here, so despite the repetition, we will present the full details of the proof in this paper. 

\begin{figure}
\centering
\includegraphics[height=60mm]{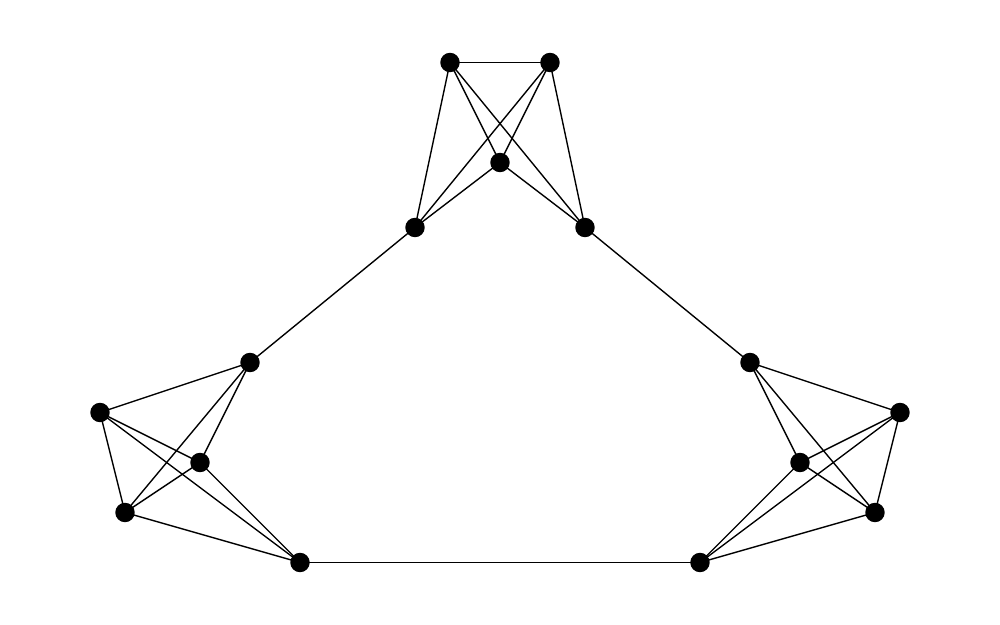}
\label{fig:15Vertices}
\caption{An example of 4-valent graph}
\end{figure}

\begin{example}\label{example:generic}
The graph $G$ with configuration $\p$ shown in Figure~\ref{fig:15Vertices} is Example 1.3 in \cite{Luo:Rigidity}. The configuration is slightly modified so that it is in general position.  It shows the order of the bound  given in Theorem~\ref{thm:generic} and Theorem~\ref{thm:general} are sharp. Let's briefly recall the construction. Take $3$ copies of complete graphs on $5$ vertices and delete one edge from each, then connect them together to form a loop as shown in the figure.  Both $r(G)$ and $r(G(\p))$ are $24$ for this graph.  This example can be easily generalized to a graph $G'$ of $5k$ vertices and configuration $\p'$ in general position such that both $r(G')$ and $r(G'(\p'))$ are $8k$. In this example, $r(G'(\p'))$ does not depend on $\p'$ as long as $\p'$ is in general position, but it should be emphasized that this is not usually the case.
\end{example}

\begin{figure}
\includegraphics[height=70mm]{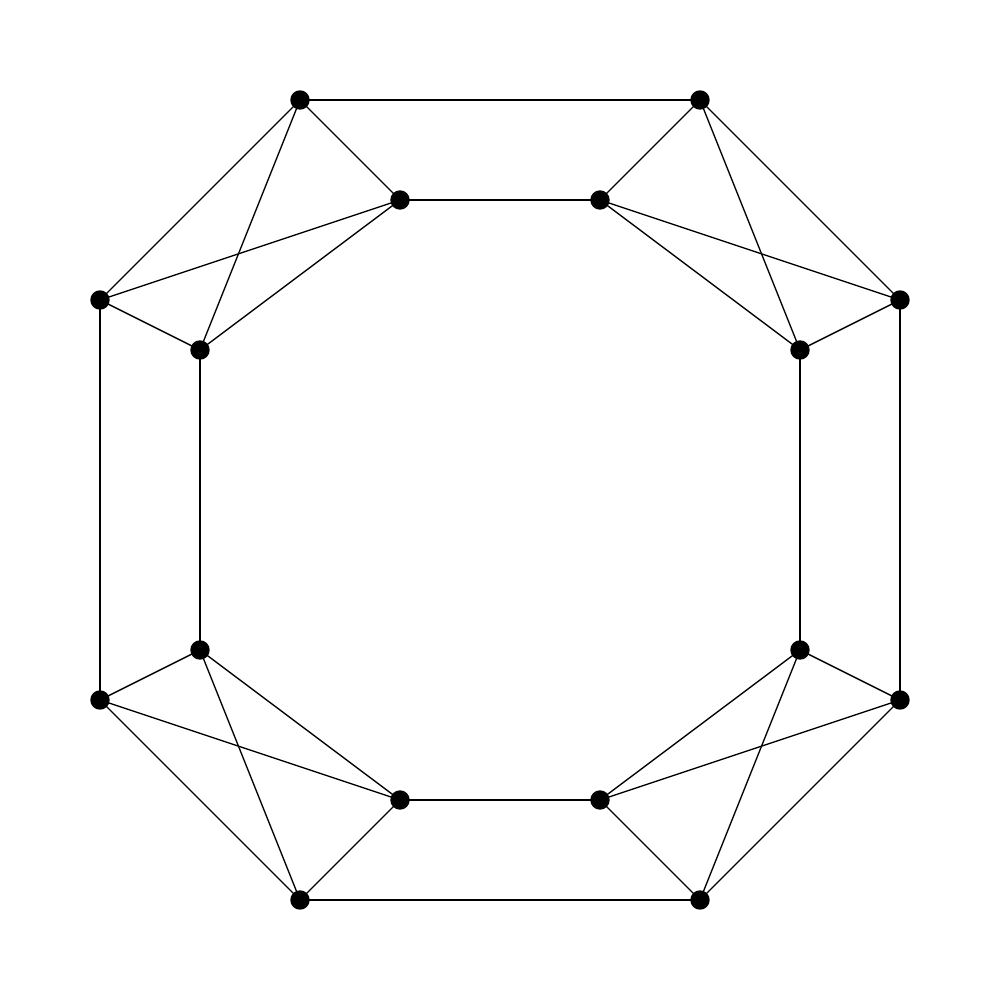}
\caption{An example of 4-valent graph that is 4-edge-connected}
\label{fig:GenericMaxS}
\end{figure}

\begin{example}
The graph $G$ shown in Figure~\ref{fig:GenericMaxS} is obtained by connecting $4$ copies of complete graphs on $4$ vertices. It is 4-edge-connected and of $r(G)=28$. This example can be easily generalized to a graph $G'$ with $4k$ vertices, such that $r(G')=7k$ if $k\geq 3$, $r(G')=13$ if $k=2$. This example shows the order of bound given in Theorem~\ref{thm:connected generic} is sharp.  
\end{example}

In order to introduce the example that shows the order of the bound given in Theorem~\ref{thm:connected general} is sharp, we first introduce the definition of Cartesian Product in graph theory.  
\begin{definition}
 the Cartesian product $G\Box H$ of graphs $G_1=(V_1, E_1)$ and $G_2=(V_2, E_2)$ is a graph such that
\begin{itemize}
\item
the vertex set of $G_1 \Box G_2$ is the Cartesian product $V_1\times V_2$; and
\item
any two vertices $(u_1,u_2)$ and $(v_1,v_2)$ are adjacent in $G_1 \Box G_2$ if and only if either
\begin{itemize}
\item[(a)]
$u_1 = v_1$ and $u_2$ is adjacent with $v_2$ in $G_2$, or
\item[(b)]
$u_2 = v_2$ and $u_1$ is adjacent with $v_1$ in $G_1$.
\end{itemize}
\end{itemize}
\end{definition}

\begin{lemma}\label{lemma:CartesianProduct}
Assume $(G_1,\p_1)$ and $(G_{2},\p_2)$ are  two connected graphs with planar configuration in general position. For any non-zero constants $a, b\in \R$, we can define a planar configuration of $G_1\Box G_2$ by 
\[\p(u_1, u_2)=a\p_1(u_1)+b\p_2(u_2).\]
If $\p$ is also a configuration in general position, then we have 
\[r(G(\p))=r(G_1(\p_1))+r(G_2(\p_2))+2(|V_1|-1)(|V_2|-1).\] 
\end{lemma}
\begin{proof}
This follows immediately from Proposition 3.18 and Proposition 3.4 in \cite{Luo:Betti}, noting $r$ was denoted by $r_1$ there.
\end{proof}

The configuration $\p: V_1\times V_2\rightarrow \R^2$ given in the above lemma, which we assume to be in general position, will be denoted by $\p_1\Box\p_2$. The notation is not perfect in that it does not suggest the dependence of $\p$ on the choices of $a, b\in \R$, but this should not bother us since the quantity $r(G(\p))$ does not rely on $a, b$ as the above lemma suggests.

\begin{figure}
\subfigure[$K_3$]{
\centering
\includegraphics[scale=0.4]{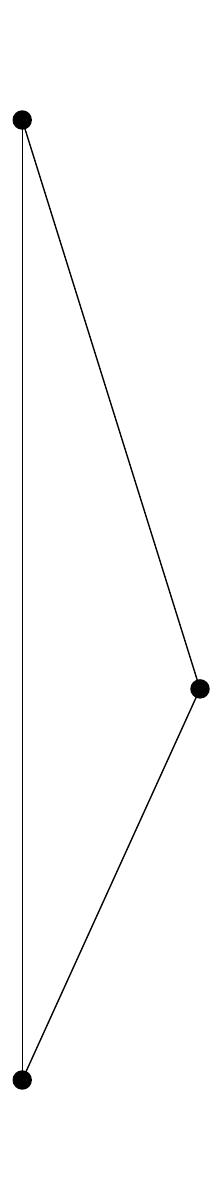}
\label{fig:K_3}
}
\subfigure[$P_6$]{
\centering
\includegraphics[scale=0.7]{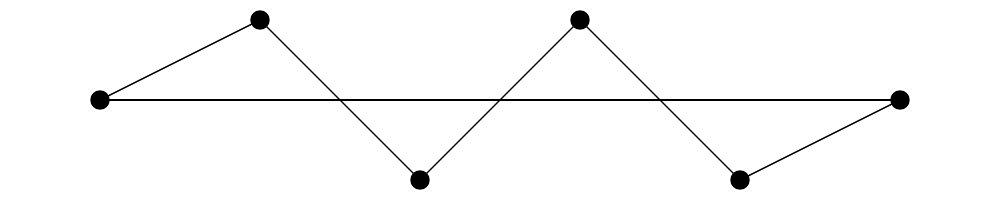}
\label{fig:P_6}
}

\subfigure[$K_3\Box P_6$]{
\centering
\includegraphics[height=60mm]{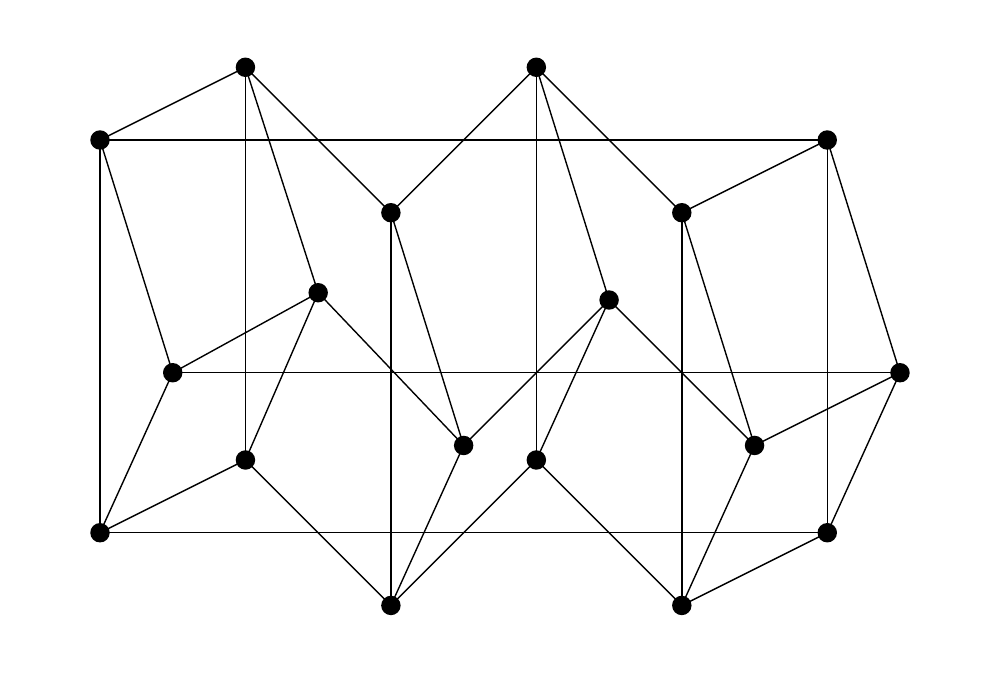}
\label{fig:K_3ProductP_6}
}
\caption{A 4-edge-connected 4-valent graph}
\label{fig:GeneralConnected}
\end{figure}

The following example is communicated to me by Prof. Bob Connelly.
\begin{example}
If we use $K_l$ to denote the the complete graph on $l$ vertices and $P_n$ to denote a $n$-gon, i.e., a regular 2-valent graph of $n$ vertices. Then $K_3\Box P_n$ is a 4-edge-connected 4-valent   graph. Figure~\ref{fig:K_3} and \ref{fig:P_6} give a planar configurations in general position for $K_3$ and $P_6$ respectively.   Figure~\ref{fig:K_3ProductP_6} shows an induced configuration in general position for $K_3\Box P_6$ in the way described in Lemma~\ref{lemma:CartesianProduct}. Denote by $G$ the product $K_3\Box P_6$ and by $\p$ the configuration illustrated in Figure~\ref{fig:K_3ProductP_6}, then according to Lemma~\ref{lemma:CartesianProduct}, \[r(G(\p))=3+6+2\times(3-1)\times(6-1)=29.\]
 This example is easily generalized to $G'=K_3\Box P_n$, which is a 4-edge-connected 4-valent graph, equipped with a planar configuration $\p'$ in general position that is induced from a configuration of $K_3$ and a configuration of $P_n$, and $r(G'(\p'))=5k-1=\dfrac{5|V|}{3}-1$. This class of graphs shows the order of the bound given in Theorem~\ref{thm:connected general} is sharp.
\end{example}

It is natural to ask the generalizations of these theorems about regular graphs of degree $d$ for $d\geq 5$. The $d=5$ version of Theorem~\ref{thm:generic} was proved in \cite{Luo:Rigidity}.  We ask the following questions. 
\begin{question}\label{question:generic}
Assume $G=(V,E)$ is a connected $d$-valent graph, $d\geq 6$, do we have 
\[r(G)\geq \dfrac{2d}{d+1}|V|-1?\]
Note that if the bound holds, then the order of it is sharp. Because we may take $k, k\geq 2,$ copies of $K_{d+1}$, delete one edge from each, then connect them to form a loop as we did in Example~\ref{example:generic}, call the resulting graph $G$. Then $r(G)=2dk=\dfrac{2d}{d+1}m$.  
\end{question}

\begin{question}
Assume $G=(V,E)$ is a connected $d$-valent graph, $d\geq 5$, and $\p: V\ra \R^2$ is a configuration in general position, do we have 
\[r(G(\p))\geq \dfrac{2d}{d+1}|V|-1?\]
Note that a positive answer to this question would imply the positive answer for Question~\ref{question:generic} as $r(G)\geq r(G(\p))$. Also, if the bound holds, the order of it is sharp. 
\end{question}

\begin{question}
Assume $G=(V,E)$ is a connected $d$-valent graph which is also $d$-edge-connected and $\p: V\ra \R^2$ is a configuration in general position, do we have
\[r(G(\p))\geq \dfrac{(2d-3)m-2(d-2)}{d-1}?\]
Again, we note that if the bound holds, then the order of it is sharp. We may consider $G=K_{d-1}\Box P_n$ with the planar configuration $\p$ in general position induced from that of $K_{d-1}$ and $P_n$ as in Lemma~\ref{lemma:CartesianProduct}. Then by Lemma~\ref{lemma:CartesianProduct}, we have \[r(G(\p))=2(d-1)-3+n+2(d-1-1)(n-1)=2dn-3n-1=\dfrac{2d-3}{d-1}m-1.\] 
Also note that the equality holds when $G=K_{d+1}$.
\end{question}

Other than edge-connectivity, there is another kind of connectivity in graph theory called {\it vertex-connectivity}.
\begin{definition}
A graph $G=(V, E)$ is called $k$-vertex-connected if for any $U\subseteq V$ with $|U|<k$, the graph $G'=(V', E')$ defined by $V'=V\backslash U, E'=E\backslash E_U$ is still connected, where $E_U$ is the subset of $E$ consisting of edges incident to some vertex in $U$.  Note that a connected graph is automatically $1$-vertex-connected.
\end{definition}

We can ask similar questions about the rank of rigidity matrix with vertex-connectivity assumptions.
\begin{question}
Assume $G=(V, E)$ is a regular $d$-valent graph that is also $k$-vertex-connected, $2\leq k\leq d$. What is the lower bound for $r(G)$? Assume $\p$ is a planar configuration in general position, what is the lower bound for $r(G(\p))$?
\end{question}

In the case of 4-valent graphs, the answer to the above question follows trivially from our results about $r(G)$ and $r(G(\p))$ when assuming (or not) edge-connectivities. We will simply state the result without proof.
\begin{corollary}
Assume $G$ is a 4-valent graph and $\p$ a planar configuration in general position.  If $G$ is $2$-vertex-connected, then $r(G)\geq r(G(\p))\geq \dfrac{8}{5}|V|-1$. If $G$ is $3$-vertex-connected or $4$-vertex-connected, then $r(G)\geq  \dfrac{7|V|-7}{4}$ and $r(G(\p))\geq \dfrac{5|V|-4}{3}$.
The order of all the above bounds are sharp.
\end{corollary} 

{\bf Acknowledgement: } I would like to thank Bob Connelly and Tara Holm for some helpful discussions.
%%%%%%%%%%%%%%%%%%%%%%%%%%%%%%%%%%%%%%%%%%%%%%%%%%
\section{\bf Proof of Theorem~\ref{thm:connected general}}\label{sec:generalConnected}
The content in this section is essentially the work in Section 5 of \cite{Luo:Betti}. 

First we briefly recall some notations we used and results we obtained in \cite{Luo:Rigidity}, the readers are referred to Section 2 of \cite{Luo:Rigidity} for more details and proofs about these (up to Lemma~\ref{lemma:disconnect2} here).
\begin{notation}
Let $G=(V,E)$ be a graph, the vertices are labeled as $v_1, v_2, ..., v_{|V|}$, the edge incident to $v_i$ and $v_j$ is denoted by $e_{ij}$. We do not distinguish between $e_{ij}$ and $e_{ji}$, but most of the time we will make the first coordinate smaller than the second one. 

We will use $\lambda(v)$ to denote the degree of $v$. When the graph $G$ needs to be emphasized (this could be the case if $v$ is the vertex of two graphs under discussion), we use $\lambda_{G}(v)$ to denote the degree of $v$ as a vertex of $G$.

We will use $E_{v_i}$ to be subset of $E$ consisting of edges incident to $v_i$. For any $U\subseteq V$, we will use $E_U$ to denote the set of edges incident to some vertex in $U$. 

The configuration $\p: V\ra \R^2$ will {\it always} be assumed to be in general position.  If $G'$ is a subgraph of $G$, then by abuse of notation, we will also use $\p$ to denote the configuration of $G'$ induced from that of $G$.
\end{notation}

\begin{definition}
Assume $G=(V, E)$ is a graph and $\p: V\ra \R^2$ is a planar configuration in general position. We use $s(G)=|E|-r(G)$ to denote the {\it number of stress} of $G$ and use $s_{\p}(G)=|E|-r(G(\p))$ to denote the {\it number of stress} of $G(\p)$. 
\end{definition}

The following four lemmas and corollary are stated and proved in \cite{Luo:Rigidity} under generic configuration assumption, but the proof could be carried over to configuration in general position without any change and we are not going to bother rewriting the proofs.

The following lemma is Lemma 2.3 in \cite{Luo:Rigidity} and called the {\it Deleting Lemma}.
\begin{lemma}\label{lemma:delete}
Given $G=(V, E)$ and $\p: V\ra \R^2$ a planar configuration in general position, assume there is a vertex $v_i$ with $\lambda(v_i)=1$ or $2$. Define $G'$ to be the graph obtained from $G$ by deleting $v_i$ and $E_{v_i}$, then $s(G')=s(G)$ and $s_{\p}(G')=s_{\p}(G)$.
\end{lemma}
\begin{corollary}\label{corollary:delete}
Given  $G=(V, E)$ and $\p: V\ra \R^2$ a planar configuration in general position, let $v_i\in V$. Define $G'$ to be the graph obtained from $G$ by deleting $v_i$ and $E_{v_i}$, then \[0\leq s(G)-s(G')\leq \max(\lambda(v_i)-2, 0),\]
and \[0\leq s_{\p}(G)-s_{\p}(G')\leq \max(\lambda(v_i)-2, 0).\]
\end{corollary}

The following two lemmas are Corollary 2.8 and Corollary 2.9 of \cite{Luo:Rigidity}.
\begin{lemma}\label{lemma:disconnect1}
Given $G(\p)$ a graph with planar configuration in general position, assume deleting one edge $e_{ij}$ increases the number of connected component of $G$ by $1$. Denote by $G'$ the graph obtained from $G$ by deleting $e_{ij}$. Then $s(G')=s(G)$ and $s_{\p}(G')=s_{\p}(G)$.
\end{lemma}

\begin{lemma}\label{lemma:disconnect2}
Given $G(\p)$ a graph with planar configuration in general position, assume deleting some two edges $e_{i_1j_1}$ and $e_{i_2j_2}$ increases the number of connected components of $G$ by $1$, but deleting any one of these two edges does not change the number of connected components.  Denote by $G'$ the graph obtained from $G$ by deleting $e_{i_1j_1}$ and $e_{i_2j_2}$. Then $s(G')=s(G)$ and $s_{\p}(G')=s_{\p}(G)$.
\end{lemma}

\begin{definition}\label{def:trim}
A graph $G=(V, E)$ is called {\it trimmed} if 
\begin{itemize}
\item each vertex of $G$ is of degree at least $3$;
\item each connected component of $G$ is $3$-edge-connected.
\end{itemize}
\end{definition}

For any graph $G=(V, E)$, we can obtain a trimmed subgraph $\tilde{G}$ of $G$ in the following way.  If $G$ is trimmed, then let $\tilde{G}=G$. If $G$ is not trimmed, then we first obtain a subgraph $G_1$ of $G$ in one of the following two ways:
\begin{itemize}
\item[(1)] If there is vertex $v_i$ of $G$ of degree $1$ or $2$, let $G_1=(V\ba\{v_i\}, E\ba E_{v_i})$.  The choice of $v_i$ may not be unique;
\item[(2)] if every vertex is of degree at least $3$, then there exists $F \subseteq E$, such that $|F|\leq 2$, and removing $F$ would increase the number of connected components of $G$, but removing any proper subset of $F$ would not increase the number of connected components of $G$, then let $G_1=(V, E\ba F)$. Again, the choice of $F$ may not be unique.
\end{itemize} 
If $G_1$ is trimmed, then we let $\tilde{G}=G_1$. Otherwise, we obtain $G_2$ as a subgraph of $G_1$ in the same way as described above. Following this process we get a sequence of graphs $G=G_0, G_1, G_2, ..., G_q$, where $G_{k+1}$ is a subgraph of $G_k$ obtained in one of two ways we just described, and $G_q$, which could be empty, is trimmed. We let $\tilde{G}$ be $G_q$.

\begin{definition}\label{def:trimming}
We call the process of getting $\tilde{G}$ from $G$ described above the {\it trimming process}.   It follows from Lemma~\ref{lemma:delete}, Lemma~\ref{lemma:disconnect1} and Lemma~\ref{lemma:disconnect2} that $s(\tilde{G})=s(G)$ and $s_{\p}(\tilde{G})=s_{\p}(G)$ for any planar configuration in general position.
\end{definition}

\begin{remark}
Evan we may have different choices for each graph $G_i$ in the middle of the trimming process,  the graph $\tilde{G}$ in fact does not depend on the choice of $G_i$'s, it is the maximal trimmed subgraph of $G$. We will not need this, hence will not bother with the proof.
\end{remark}

We make the following definition so that the statements and proofs in the rest of this section could be made more concise.
\begin{definition}\label{def:A_4}
We call a graph $G=(V, E)$ is {\it of type $A_{4}$} if 
\begin{itemize}
\item each vertex of $G$ is of degree $3$ or $4$;
\item each connected component of $G$ has at least one vertex of degree $3$;
\item each connected component of $G$ is $3$-edge-connected.
\end{itemize}
\end{definition}

The following proposition is a special case of Proposition 5.7 of \cite{Luo:Betti}.
\begin{proposition}\label{prop:s for A_4}
Assume $G=(V, E)$ is a graph of type $A_{4}$ and $\p: V\ra \R^2$ a planar configuration in general position,  then 
\begin{equation}\label{eq:s for A_4}
s_{\p}(G)\leq \dfrac{n_{4}(G)}{3}+c(G),\end{equation}
where $n_{4}(G)=\{v_i\in V\big| \lambda(v_i)=4\}$ is the number of vertices of degree $4$, $c(G)$ is the number of connected components of $G$.
\end{proposition}
\begin{proof}
We are going to use induction on the size of $|V|$. The graph of type $A_4$ with lease number of vertices is the empty graph, inequality \eqref{eq:s for A_4} holds trivially in this case. 

Now we consider $G=(V, E)$ with $|V|=m\geq 1$ and assume \eqref{eq:s for A_4} holds for $|V|<m$. If $c(G)>1$, then each connected compnent of $G$ is also of type $A_4$ and of less vertices. So the induction hypothesis says \eqref{eq:s for A_4} holds for each connected component. We may then add them up to show \eqref{eq:s for A_4} holds for $G$ as well. Now we assume $G$ is connected. 

Since $G$ is of type $A_4$, there exists a vertex $v_t$ of degree $3$. Define $G'=(V', E')$ by $V'=V\ba\{v_t\}$ and $E'=E\ba E_{v_i}$. By Corollary~\ref{corollary:delete}, we have $s_{\p}(G)\leq s_{\p}(G')+1$.  

Apply the trimming process to $G'$ to obtain $\tilde{G'}$, then $s_{\p}(\tilde{G'})=s_{\p}(G')$ and $\tilde{G'}$ is of type $A_4$. 

Assume $c(\tilde{G'})=a, \displaystyle{\tilde{G'}=\bigsqcup_{i=1}^a G_i}$ and $G_i=(V_i, E_i)$. Since $G$ is $3$-edge-connected, for each $G_i$, there must be at least $3$ vertices $v$ in $V_i$, such that $\lambda_{G}(v)=4$ and $\lambda_{G_i}(v)=3$. So 
\[\sum_{i=1}^{a}n_{4}(G_i)\leq n_{4}(G)-3a.\]
Since $\tilde{G'}$ is of type $A_4$ and of less vertices than $G$, by the induction hypothesis we have 
\[s_{\p}(\tilde{G'})\leq \dfrac{n_{4}(\tilde{G'})}{3}+c(\tilde{G'}).\]
So 
\[
\begin{array}{cl}
s_{\p}(G)&\leq s_{\p}(G')+1= s_{\p}(\tilde{G'})+1\\[0.5ex]
&\leq \dfrac{n_{4}(\tilde{G'})}{3}+c(\tilde{G'})+1\\[1.2ex]
&=\dfrac{1}{3}\ds{\sum_{i=1}^a n_{4}(G_i)} + a+1\\
&\leq \dfrac{1}{3}(n_{4}(G)-3a)+a+1\\[1.2ex]
&= \dfrac{n_{4}(G)}{3}+1=\dfrac{n_{4}(G)}{3}+c(G).
\end{array}
\]
This completes the induction step.
\end{proof}

\begin{proof}[Proof of Theorem~\ref{thm:connected general}]
Pick one edge $e_{ij}\in E$ and form a new graph $G'=(V, E\ba\{e_{ij}\})$. By Corollary~\ref{corollary:delete} we have $s_{\p}(G)\leq s_{\p}(G')+1$. The graph $G'$ is of type $A_4$, so by Proposition~\ref{prop:s for A_4}, we have $s_{\p}(G')\leq \dfrac{|V|-2}{3}+1$, hence $s_{\p}(G)\leq \dfrac{|V|+4}{3}$. Therefore 
\[r(G(\p))=|E|-s_{\p}(G)\geq 2|V|-\dfrac{|V|+4}{3}=\dfrac{5|V|-4}{3}.\] 
\end{proof}

%%%%%%%%%%%%%%%%%%%%%%%%%%%%%%%%%%%%%%%%%%%%%%%%%%%%%%
\section{\bf Proof of Theorem~\ref{thm:generic} and Theorem~\ref{thm:general}}\label{sec:genericAndGeneral}
\begin{proof}[Proof of Theorem~\ref{thm:general}]
If $G$ is the empty graph, the theorem holds trivially. Now we may assume $|V|\geq 5$.
 If $G$ is $4$-edge-connected, then by Theorem~\ref{thm:connected general}, we have \[r(G(\p))\geq \dfrac{5|V|-4}{3}\geq \dfrac{8|V|}{5}-1\] because $|V|\geq 5$.
 
 If $G$ is not $4$-edge-connected,  then there exists $F\subseteq E$ with $|F|\leq 3$, such that the graph $G'=(V, E\ba F)$ is  disconnected. We may assume $F$ has the least cardinality among the subsets of $E$ with this property. We observe $|F|$ cannot be $3$ or $1$, because the sum of degrees of the vertices of one connected component of $G'$ has to be an even number. So $|F|=2$.
Now apply the trimming process, which was defined in Definition~\ref{def:trimming}, to $G$, we get $\tilde{G}$.  Then $\tilde{G}$ is a proper subgraph of $G$ and is of type $A_4$. 

We put the connected components of $\tilde{G}$ into two categories: the type of four vertices and the type of at least five vertices. Assume $\tilde{G}$ decomposes into disjoint union of connected components
\[\tilde{G}=(\bigsqcup_{i=1}^{a} G_i)\bigsqcup (\bigsqcup_{i=1}^{b}H_i),\]
where each $G_i$ is of four vertices, hence has to be the complete graph on four vertices, each $H_{i}$ is of at least five vertices and of type $A_4$.  Assume $H_i$ has $m_i$ vertices, then $n_{4}(H_i)\leq m_{i}-2$, so by Proposition~\ref{prop:s for A_4}, we have
\[s_{\p}(H_{i})\leq \dfrac{m_i-2}{3}+1.\]
So 
\begin{equation*}
\begin{array}{cl}
s_{\p}(G)&=\ds{s_{\p}(\tilde{G})=\sum_{i=1}^{a}s_{\p}(G_i) +\sum_{i=1}^{b}s_{\p}(H_i)}\\
&\leq a + \ds{\sum_{i=1}^{b}(\dfrac{m_i-2}{3}+1)}\\
&\ds{= a+ \dfrac{b}{3}+\dfrac{1}{3}\sum_{i=1}^{b}m_i }\\
&\leq  a+\dfrac{b}{3} + \dfrac{1}{3}(|V|-4a)\\
&=\dfrac{1}{3}|V|+\dfrac{b}{3}-\dfrac{a}{3},
\end{array}
\end{equation*} 
and because 
\begin{equation*}
\begin{array}{cl}
&(\dfrac{1}{3}|V|+\dfrac{b}{3}-\dfrac{a}{3})-(\dfrac{2}{5}|V|+1)\\
=&-\dfrac{1}{15}|V|+\dfrac{b}{3}-\dfrac{a}{3}-1\\[1.5ex]
\leq &-\dfrac{1}{15}(|V|-5b)\leq 0
\end{array}
\end{equation*}
we immediately see that $s_{\p}(G)\leq \dfrac{2}{5}|V|+1$, so \[r(G(\p))\geq 2|V|-(\dfrac{2}{5}|V|+1)=\dfrac{8}{5}|V|-1.\]
\end{proof}

\begin{proof}[Proof of Theorem~\ref{thm:generic}]
This follows trivially from Theorem~\ref{thm:general} because  we always have $r(G)\geq r(G(\p))$.
\end{proof}

%%%%%%%%%%%%%%%%%%%%%%%%%%%%%%%%%%%%%%%%%%%%%%%%%%%%%%
\section{\bf Proof of Theorem~\ref{thm:connected generic}}\label{sec:genericConnected}
The proof uses a similar idea with that in the proof of Theorem~\ref{thm:connected general}. First we recall two results about generic rigidity from Section 2 of \cite{Luo:Rigidity}.
\begin{lemma}\label{lemma:disconnect3}
Given a graph $G=(V, E)$, assume $F\subseteq E$ with $|F|=3$ and the edges in $F$ do not share a common vertex. let $G'=(V, E\ba F)$. If removing $F$ increases the number of connected components of $G$, but removing any proper subset of $F$ will not do so, then $s(G')=s(G)$. 
\end{lemma}
\begin{proof}
This follows immediately from Lemma 2.7 in \cite{Luo:Rigidity}.
\end{proof}

\begin{lemma}\label{lemma:one extension}
Given $G=(V, E)$, assume there is a vertex $v_t$ of degree $3$ and the three vertices adjacent to $v_t$ are $v_i, v_j$ and $v_k$. Assume $e_{ij}\notin E$. We define a new graph $G'=(V', E')$ by $V'= V\ba \{v_t\}$ and $E'=(E\cup\{e_{ij}\})\ba E_{v_t}$. Then $s(G)\leq s(G')$.
\end{lemma}
\begin{proof}
This is Proposition 2.12 in \cite{Luo:Rigidity}.
\end{proof}

\begin{definition}\label{def:generic trim}
A graph $G=(V, E)$ is called {\it generically trimmed} if 
\begin{itemize}
\item it is trimmed (see Definition~\ref{def:trim});
\item if there exists $F\subseteq E$, such that $|F|=3$ and removing $F$ increases the number of connected components of $G$, then all the edges in $F$ share a common vertex.
\end{itemize}
\end{definition}

Similar to the trimming process, we may define the {\it generic trimming process}.
For any graph $G=(V, E)$, we can obtain a generically trimmed subgraph $\tilde{G}^g$ of $G$ in the following way.  If $G$ is generically trimmed, then let $\tilde{G}^g=G$. If $G$ is not generically trimmed, then we first obtain a subgraph $G_1$ of $G$ in one of the following three ways:
\begin{itemize}
\item[(1)] If there is vertex $v_i$ of $G$ of degree $1$ or $2$, let $G_1=(V\ba\{v_i\}, E\ba E_{v_i})$.  The choice of $v_i$ may not be unique;
\item[(2)] if every vertex is of degree at least $3$, and there exists $F \subseteq E$, such that $|F|\leq 2$, and removing $F$ would increase the number of connected components of $G$, but removing any proper subset of $F$ would not increase the number of connected components of $G$, then let $G_1=(V, E\ba F)$. The choice of $F$ may also not be unique.
\item[(3)] if $G$ is trimmed, and there exists $F\subseteq E$ with $|F|=3$, such that the vertices in $F$ do not share a common vertex and removing $F$ will increase the number of connected components of $G$, then let $G_1=(V, E\ba F)$. Again, the choice of $F$ may not be unique. 
\end{itemize} 
If $G_1$ is generically trimmed, then we let $\tilde{G}^g=G_1$. Otherwise, we obtain $G_2$ as a subgraph of $G_1$ in the same way as described above. Following this process we get a sequence of graphs $G=G_0, G_1, G_2, ..., G_q$, where $G_{k+1}$ is a subgraph of $G_k$ obtained in one of three ways we just described, and $G_q$, which could be empty, is generically trimmed. We let $\tilde{G}^g$ be $G_q$.

\begin{definition}\label{def:generic trimming}
We call the process of getting $\tilde{G}^g$ from $G$ described above the {\it generic trimming process}.   It follows from Lemma~\ref{lemma:delete}, Lemma~\ref{lemma:disconnect1}, Lemma~\ref{lemma:disconnect2} and Lemma~\ref{lemma:disconnect3} that $s(\tilde{G}^g)=s(G)$.
\end{definition}

\begin{definition}
We call a graph $G=(V, E)$ is {\it of type $B_{4}$} if 
\begin{itemize}
%\item it is of type $A_4$ (see Definition~\ref{def:A_4});
\item each vertex of $G$ is of degree $3$ or $4$;
\item each connected component of $G$ has at least {\it four} vertices of degree $3$;
\item it is generically trimmed.
\end{itemize}
\end{definition}

\begin{lemma}\label{lemma:K_4 and B_4}
Given $G=(V, E)$ is a connected graph of type $B_4$, assume $v_t\in V$ is of degree $3$ and the three vertices adjacent to $v_t$ are $v_{i_1}, v_{i_2}, v_{i_3}$. Denote by $K_4=(V', E')$ the complete graph on $v_t, v_{i_1}, v_{i_2}$ and $v_{i_3}$. If $K_4$ is a subgraph of $G$, then $G=K_4$.
\end{lemma}
\begin{proof}
If $G\neq K_4$, then $F=\{e\in E\ba E'\big| e \mbox{\ is incident to\ }v_{i_1}, v_{i_2} \mbox{\ or\ } v_{i_3}\}$ is not empty. Deleting $F$ from $G$ would disconnect $G$. Because $G$ is generically trimmed, we must have $|F|=3$ and there exists $v_{x}\in V\ba V'$, such that $\{e_{i_1x}, e_{i_2x}, e_{i_3x}\}\subseteq E$. 
Because $G$ is of type $B_4$, either $\lambda(v_x)=3$ or $\lambda(v_x)=4$. If $\lambda(v_x)=3$, then $G$ only have five vertices: $v_{t}, v_{i_1}, v_{i_2}, v_{i_3}, v_x$, and only $2$ of them are of degree $3$, this contradicts the assumption that $G$ is of type $B_4$.  If $\lambda(v_x)=4$, assume $v_y\in V\ba V'$ is the other vertex adjacent to $v_x$. Then deleting $e_{xy}$ will disconnect $G$, contradicting the assumption that $G$ is generically trimmed. 
So $G=K_4$.  
\end{proof}

\begin{proposition}\label{prop:s for B_4}
If $G=(V, E)$ is of type $B_4$, then 
\begin{equation}\label{eq:s for B_4}
s(G)\leq \dfrac{n_4(G)}{4}+c(G),
\end{equation}
where $n_4(G)$ is the number of vertices of degree $4$ in $G$, $c(G)$ is the number of connected components of $G$. 
\end{proposition}
\begin{proof}
We are going to use induction on $|V|$. The base case is $|V|=0$, \eqref{eq:s for B_4} holds trivially in this case. Now assume $|V|=m>0$ and \eqref{eq:s for B_4} holds for graphs with less than $m$ vertices. If $c(G)>1$, then each connected component of $G$ is of type $B_4$ and of less vertices, so \eqref{eq:s for B_4} holds for each connected component by the induction hypothesis, we may then add them to show \eqref{eq:s for B_4} holds for $G$ as well.  Now we assume $G$ is connected. 

Since $G$ is of type $B_4$, there exists a vertex $v_t$ of degree $3$. Let $G_1=(V_1, E_1)=(V\ba\{v_t\}, E\ba E_{v_t})$. Then by Corollary~\ref{corollary:delete}, we have $s(G)\leq s(G_1)+1$. 

Case 1: $G_1$ is of type $B_4$. Assume the three vertices adjacent to $v_t$ are $v_{i_1}, v_{i_2}, v_{i_3}$.
\begin{itemize}
\item[(a)] If $\{e_{i_1i_2}, e_{i_1i_3}, e_{i_2i_3}\}\subseteq E$, then by Lemma~\ref{lemma:K_4 and B_4} we have $G=K_4$ and we can check \eqref{eq:s for B_4} holds in this case.  
\item[(b)] If $\{e_{i_1i_2}, e_{i_1i_3}, e_{i_2i_3}\}\not\subseteq E$, then without loss of generality we may assume $e_{i_1i_2}\notin E$. Define $G_2=(V_2, E_2)$ by $V_2= V\ba \{v_t\}$ and $E_2=(E\cup\{e_{i_1i_2}\})\ba E_{v_t}$. Then by Lemma~\ref{lemma:one extension}, we have $s(G)\leq s(G_2)$. Because both $G$ and $G_1$ are of type $B_4$, $G_2$ must be of type $B_4$ as well. Since $|V_2|<|V|$, by the induction hypothesis, we have $s(G_2)\leq \dfrac{n_4(G_2)}{4}+1$. So $s(G)\leq s(G_2)\leq \dfrac{n_4(G)}{4}+1$.
\end{itemize}

Case 2: $G_1$ is not of type $B_4$. 

In this case we apply the generic trimming process to $G_1$ to obtain $\tilde{G_1}^g$, then $s(G_1)=s(\tilde{G_1}^g)$. Assume $c(\tilde{G_1}^g)=a$ and $\ds{\tilde{G_1}^g=\bigsqcup_{i=1}^{a}H_i.}$ 

We claim in each $H_i$, there are at least $4$ vertices $v$ such that $\lambda_{H_i}(v)=3$ and $\lambda_{G}(v)=4$.  If not, because $G$ is $3$-edge-connected, there exist exactly three vertices $v_{j_1}, v_{j_2}, v_{j_3}$ such that $\lambda_{H_{i}}(v_{j_k})=3, \lambda_{G}(v_{j_k})=4$ for $1\leq k\leq 3$. Assume the edges in $G$ but not in $\tilde{G_1}^g$ and incident to some $v_{j_k}$ are $e_{xj_1}, e_{yj_2}$ and $e_{zj_3}$. Because $G$ is of type $B_4$, hence generically trimmed, we must have $x=y=z$. So if $v_x\in V_1$, then $v_x$ is a vertex of $\tilde{G_1}^g$, but this is not the case, so $v_x\notin V_1$, hence $v_x=v_t$.  Then $H_i=G_1$, this contradicts the assumption that $G_1$ is not of type $B_4$.  So we have proved the claim. 

So each $H_i$, hence $\tilde{G_1}^g$, is of type $B_4$, by the induction hypothesis, we have \[s(\tilde{G_1}^g)\leq \dfrac{n_4(\tilde{G_1}^g)}{4}+c(\tilde{G_1}^g).\]  So 
\[
\begin{array}{cl}
s(G)&\leq s(G_1)+1= s(\tilde{G_1}^g)+1\\[0.5ex]
&\leq \dfrac{n_{4}(\tilde{G_1}^g)}{4}+c(\tilde{G}^g)+1\\[0.6ex]
&\leq \dfrac{1}{4}(n_{4}(G)-4a)+a+1\\[1.2ex]
&= \dfrac{n_{4}(G)}{4}+1=\dfrac{n_{4}(G)}{4}+c(G).
\end{array}
\]
This completes the induction steps. 
\end{proof}

\begin{proof}[Proof of Theorem~\ref{thm:connected generic}]
Define $G_1=(V_1, E_1)$ by $V_1=V\ba \{v_1\}$ and $E_1=E\ba E_{v_1}$. Then by Corollary~\ref{corollary:delete}, $s(G)\leq s(G_1)+2$. Apply the generic trimming process to $G_1$, we obtain $\tilde{G_1}^g$, then $s(G_1)=s(\tilde{G_1}^g)$. Assume $c(\tilde{G_1}^g)=a$ and $\ds{\tilde{G_1}^g=\bigsqcup_{i=1}^{a}H_i}.$ For each $H_i$, there must exist at least $4$ vertices of degree $3$, because $G$ is $4$-edge-connected. So each $H_i$, hence $\tilde{G_1}^g$, is of type $B_4$. By Proposition~\ref{prop:s for B_4}, we have 
\[s(\tilde{G_1}^g)\leq \dfrac{n_4(\tilde{G_1}^g)}{4}+a.\]
So 
\[
\begin{array}{cl}
s(G)&\leq s(G_1)+2= s(\tilde{G_1}^g)+2\\[0.5ex]
&\leq \dfrac{n_4(\tilde{G_1}^g)}{4}+a+2\\
&\leq \dfrac{1}{4}(|V|-1-4a)+a+2\\[1.5ex]
&=\dfrac{|V|+7}{4}
\end{array}
\] 
Hence
\[r(G)= 2|V|-s(G)\geq \dfrac{7|V|-7}{4}.\]
\end{proof}

\end{document}